\documentclass[a4paper]{amsart}
\usepackage{epsfig}
\usepackage{amsthm,amsfonts}
\usepackage{amssymb,graphicx,color}
\usepackage[all]{xy}
\usepackage{cancel} 
\usepackage{verbatim}
\usepackage{hyperref}
\usepackage{enumitem}

\newtheorem{innercustomthm}{Theorem}
\newenvironment{customthm}[1]
  {\renewcommand\theinnercustomthm{#1}\innercustomthm}
  {\endinnercustomthm}

\newtheorem{theorem}{Theorem}[section]
\newtheorem*{theorem*}{Theorem}
\newtheorem{lemma}[theorem]{Lemma}
\newtheorem{corollary}[theorem]{Corollary}
\newtheorem{proposition}[theorem]{Proposition}
\newtheorem{claim}{Claim}[theorem]

\theoremstyle{definition}
\newtheorem{definition}[theorem]{Definition}
\newtheorem{example}[theorem]{Example}

\newtheorem{question}{Question}

\theoremstyle{remark}
\newtheorem{remark}[theorem]{Remark}

\newcommand{\R}{\mathbb{R}}
\newcommand{\C}{\mathbb{C}}

\begin{document}

\title[On smoothness, tangent cones, and the metric geometry of sets]
{On smoothness, tangent cones, and the metric geometry of definable sets}

\author[A. G. Rocha]{Andr\'e Gadelha Rocha}

\author[J. E. Sampaio]{Jos\'e Edson Sampaio}
\address{Andr\'e Gadelha Rocha and Jos\'e Edson Sampaio:  
              Departamento de Matem\'atica, Universidade Federal do Cear\'a,
	      Rua Campus do Pici, s/n, Bloco 914, Pici, 60440-900, 
	      Fortaleza-CE, Brazil. \newline   
              E-mail: {\tt edsonsampaio@mat.ufc.br}                    
}

\thanks{The third named author was partially supported by CNPq-Brazil grant 303375/2025-6. This work was supported by the Serrapilheira Institute (grant number Serra -- R-2110-39576). 
}
\keywords{Lipschitz geometry, analytic sets, regularity, smoothness, tangent cones}
\subjclass[2010]{14B05; 32S50}

\begin{abstract}
In this paper, we present several definitive characterizations of the $C^1$ smoothness of definable sets in terms of their tangent cones and some other metric properties. In particular, we recover some of the beautiful characterizations presented by Ghomi and Howard (2014) and by Kurdyka, Le Gal, and Nhan (2018). For instance, we prove that for any $X\subset \R^n$ that is a locally closed $d$-dimensional definable set in an o-minimal structure, the following items are equivalent:
\begin{enumerate}
 \item $X$ is Lipschitz normally embedded (LNE), $C_3(X,p)$ is a $d$-dimensional linear subspace for any $p\in X$ and depends continuously on $p$;
 \item For each $p\in X$, $X$ is Lipschitz regular at $p$ and $C_4(X,p)$ is a $d$-dimensional linear subspace;
 \item $X$ is a topological manifold and for each $p\in X$, $X$ is LNE at $p$ and $C_4(X,p)=C_3(X,p)$;
 \item $X$ is a topological manifold, and $C_5(X,p)$ is a $d-$dimensional subset for any $p\in X$;
 \item $X$ is $C^1$ smooth.
\end{enumerate}
\end{abstract}

\maketitle
\tableofcontents

\section{Introduction}

A fundamental problem in singularity theory is to know how simple the topology of complex analytic sets is. For example, {\it does topological regularity imply analytic regularity?} In general, this does not occur, but Mumford in \cite{Mumford:1961} showed a result in this direction, which can be formulated as follows: \emph{a topologically regular complex surface in $\C^3$,  with isolated singularities, is smooth}.

The second author, in his thesis \cite{Sampaio:2015} (see also \cite{Sampaio:2016} and \cite{BirbrairFLS:2016}), proved a version of Mumford's Theorem. He showed that if a complex analytic set is Lipschitz regular at $p$, then it is smooth at $p$. A definitive metric characterization of the smoothness of complex analytic sets was established in \cite{Sampaio:2025a}, showing that {\it a complex analytic set is smooth at $p$ if and only if it is $\alpha$-H\"older regular at $p$ for all $\alpha\in (0,1)$}. Recall that a subset $X\subset\R^{n}$ is called {\bf Lipschitz regular at $p\in X$} if there is an open neighborhood $U$ of $p$ in $X$ that is bi-Lipschitz homeomorphic to a Euclidean ball. The notion of $\alpha$-H\"older regularity is defined similarly.

In the real setting, for each positive integer $d$, we have the following question: {\it Let $X\subset \mathbb R^n$ be a $d$-dimensional real analytic set that is Lipschitz regular at $0\in X$. Is it true that $X$ is $C^1$ smooth at $0$?}

In the paper \cite{Sampaio:2021}, it was proved that the above question has a positive answer if and only if $d=1$. In particular, 
it was proved that for $d>1$, $X=\{(x_1,...,x_{d+1})\in \R^{d+1};\, x_3^3=x_1^3+x_2^3\}$ is bi-Lipschitz homeomorphic to $\R^d$ but fails to be $C^1$ smooth.

Thus, in order to obtain $C^1$ smoothness of real analytic sets, we must require more than Lipschitz regularity. 

In this article, we address the problem of characterizing $C^1$ submanifolds in the context of o-minimal structures in terms of their tangent cones, which are defined in Definition \ref{def:tg_cones}. To know more about o-minimal geometry, see, for instance, \cite{Coste:1999} and \cite{Dries:1998}

Characterizing $C^1$ submanifolds of $\R^n$ in terms of their third and fifth tangent cones has already been studied by many authors,
see, for example, \cite{BigolinG:2012,GhomiH:2014,Gluck:1968,KurdykaGN:2018,Tierno:2000}, and the survey \cite{BigolinG:2000}.  
For instance, it was proved in \cite[Theorem 4.7]{KurdykaGN:2018} the following very interesting characterization:
\begin{theorem}\label{thm:KGN}
Let $X\subset \R^n$ be a locally closed definable set in an o-minimal structure. Then, $X$ is $C^1$ smooth if and only if $C_5(X,p)=C_3(X,p)$ for any $p\in X$.
\end{theorem}

Another interesting result is the following characterization, for sets of codimension one, which follows directly from the results proved in \cite{GhomiH:2014}:
\begin{theorem}\label{thm:GhomiH}
    Let $X\subset \R^n$ be an $(n-1)$-dimensional definable set in an o-minimal structure. Then $X$ is $C^1$ smooth if and only if $X$ is a topological manifold and $C_3(X,p)$ is an $(n-1)$-dimensional hyperplane for each $p\in X$ that depends continuously on $p$.
\end{theorem}

For complex analytic sets, the authors in \cite{BirbrairFLS:2016} presented a characterization of smoothness in terms of the third tangent cone and under the condition that the set is LNE. They proved that {\it a pure dimensional complex analytic set $X\subset \C^n$ is smooth around $p$ if and only if $X$ is LNE at $p$ and $C_3(X,p)$ is a linear subspace}. Recall that a set $X\subset \R^m$ is {\bf Lipschitz normally embedded} ({\bf LNE}) {\bf at} $p\in X$ if there are a neighborhood $U\subset \R^m$ of $p$ and a constant $C\geq 1$ such that $d_{X\cap U}(x,y)\leq C\|x-y\|$ for all $x,y\in X\cap U$, where $d_{X\cap U}$ denotes the geodesic distance on $X\cap U$.

In this paper, we generalise Theorems \ref{thm:KGN} and \ref{thm:GhomiH} and present results as the characterisation of \cite{BirbrairFLS:2016} in the much more general context of sets that are locally definable in o-minimal structures. More precisely, we prove the following result.
\begin{customthm}{\ref*{thm:main_thm_complete}}
Let $X\subset \R^n$ be a $d$-dimensional definable set in an o-minimal structure. Then the following items are equivalent.
\begin{enumerate}
 \item  $X$ is locally closed and LNE at any $p\in X$, and $C_3(X,p)$ is a $d$-dimensional linear subspace for any $p\in X$ and depends continuously on $p$;
 \item For each $p\in X$, $X$ is Lipschitz regular at $p$ and $C_4(X,p)$ is a $d$-dimensional linear subspace;
 \item $X$ is a topological manifold and for each $p\in X$, $X$ is LNE at $p$ and $C_4(X,p)=C_3(X,p)$;
 \item $X$ is a topological manifold and $C_5(X,p)$ is a $d$-dimensional linear subspace set for any $p\in X$;
 \item $X$ is $C^1$ smooth.
\end{enumerate}
\end{customthm}

We prove Theorem \ref{thm:main_thm_complete} in Section \ref{sec:reg_definable}. This theorem follows from Theorems \ref{thm:main_thm_one}, \ref{thm:main_thm}, \ref{thm:c3=c4_regularity}, and \ref{thm:dim_c5_general}.

As a consequence, we obtain \cite[Theorem 4.7]{KurdykaGN:2018}, which was stated above as Theorem \ref{thm:KGN} (see Corollary \ref{cor:c3_c5}).

The generalization of Theorem \ref{thm:GhomiH} is presented in Corollary \ref{cor:gen_thm:GhomiH}.

When we deal with real analytic sets, we obtain the following stronger result.

\begin{customthm}{\ref*{thm:reg_analytic_sets}}
Let $X\subset \R^n$ be a $d$-dimensional real analytic set. Then the following items are equivalent:
\begin{enumerate}
    \item $X$ is LNE at any $p\in X$ and $C_4(X,p)$ is a $d$-dimensional linear subspace for all $p\in X$;
    \item $C_5(X,p)$ is a $d$-dimensional linear subspace for all $p\in X$;
    \item $X$ is $C^1$ smooth.
\end{enumerate}
\end{customthm}

In the special case that $X$ is the zero set of a harmonic function $f\colon U\subset \R^2\to \R$, it is easy to prove that $X$ is smooth if and only if it is a topological manifold. So, a natural question regarding the smoothness of a set that is the zero set of a harmonic function is the following. 

\begin{question}\label{question:harmonic_sets}
Let $X$ be a set that is the zero set of a harmonic function $f\colon U\subset \R^n\to \R$. If $X$ is Lipschitz regular at $p\in X$, is $X$ is $C^1$ smooth at $p$?
\end{question}
Unfortunately, this question has a negative answer (see Example \ref{exam:hamornic_set_nonsmooth}).

In the parametrised case, Nu\~no-Ballesteros and Mendes in \cite{Nuno-BallesterosM:2020} proved that {\it for a germ of a real analytic mapping $F\colon (\R^n,0)\to (\R^k,0)$ that is a lipeomorphism, its image $X=Im(F)$ is smooth in $0$}.

So, we have the following natural question:

\begin{question}\label{question:parametrized}
Let $F\colon (\R^n,0)\to (\R^k,0)$ be the germ of a real analytic mapping and $X=Im(F)$. If $X$ is Lipschitz regular at $0$,  is $X$ $C^1$ smooth at $0$?
\end{question}

Unfortunately, this question has a negative answer, as given in the following example.

\begin{example}
Let $F\colon \R^2\to \R^3$ be the mapping given by $F(x,y)=(x,y^3,(x+y)^3)$. Then $X=Im(F)$ is Lipschitz regular at $0$, but $X$ is not smooth at $0$.
\end{example}

The details explaining why $X=Im(F)$ is Lipschitz regular at $0$ and is  not smooth at $0$ are presented in Proposition \ref{prop:parametrized_fails}.

Another result that gives a beautiful characterization of smoothness of complex analytic sets is the renowned Theorem of Nobile, proved by Nobile in 1975, which states that {\it a pure dimensional complex analytic set $X$ is analytically smooth if and only if its Nash mapping $\eta: \mathcal{N}(X) \to X$ is an analytic isomorphism}. See the definition of the Nash mapping in Subsection \ref{subsec:nash_map}. Recently, the second author in \cite{Sampaio:2025b} proved the real version of the Theorem of Nobile. Indeed, he proved a much more general result that holds in the setting of definable sets in an o-minimal structure, as stated below.

\begin{theorem}\label{thm:c11_smooth}
Let $X\subset \R^n$ be a locally closed set that is pure $d$-dimensional and locally definable in an o-minimal structure on $\R$. Then, for a fixed nonnegative integer number $k$, $X$ is $C^{k+1,1}$ smooth if and only if the mapping $\eta\colon\mathcal{N}(X)\to X$ is a homeomorphism such that $\eta^{-1}$ is $C^{k,1}$ smooth and $C_3(X,p)=C_4(X,p)$ for any $p\in X$.  Moreover, $X$ is $C^{\infty}$ smooth if and only if the mapping $\eta\colon\mathcal{N}(X)\to X$ is a $C^{\infty}$ diffeomorphism and $C_3(X,p)=C_4(X,p)$ for any $p\in X$. 
\end{theorem}

In fact, Theorem \ref{thm:c11_smooth} holds true even for log-Lipschitz regularity (see \cite[Theorem 3.5]{Sampaio:2025b}); however, surprisingly (and in contrast to the paper \cite{Sampaio:2025a}), an example was presented in \cite{Sampaio:2025b} (see \cite[Example 3.8]{Sampaio:2025b}) that shows Theorem \ref{thm:c11_smooth} does not hold true when we only require that $\eta$ is a homeomorphism such that $\eta^{-1}$ is $C^{0,\alpha}$ smooth for all $\alpha\in (0,1)$. 
Here, we show that such a version of Theorem  \ref{thm:c11_smooth} holds true if we add the extra condition that the set $X$ is LNE at each $p\in X$.

\bigskip

\noindent{\bf Acknowledgements}. We would like to thank Guillermo Sanchis Pe\~{n}afort for helping us find the example of Proposition \ref{prop:parametrized_fails}. We would also like to thank Alexandre Fernandes, Euripedes C. da Silva, Juan Jos\'e Nu\~{n}o Ballesteros, and Lev Birbrair for their interest in this research.

\section{Preliminaries}\label{sec:preliminaries}

\subsection{O-minimal structures}
This subsection is closely related to the section with the same title in \cite{Sampaio:2025b}.

\begin{definition}\label{definivel}
A structure on $\R$ is a collection $\mathcal{S}=\{\mathcal{S}_n\}_{n\in \mathbb{Z}_{>0}}$ where each $\mathcal{S}_n$ is a set of subsets of $\R^n$, satisfying the following axioms:
\begin{itemize}
\item [1)] All algebraic subsets of $\R^n$ are in $\mathcal{S}_n$;
\item [2)] For every $n$, $\mathcal{S}_n$ is a Boolean subalgebra of the powerset of $\R^n$;
\item [3)] If $A\in \mathcal{S}_m$ and $B \in S_n$, then $A \times B \in \mathcal{S}_{m+n}$.
\item [4)] If $\pi \colon \R^{n+1} \to \R^n$ is the projection on the first $n$ coordinates and $A\in \mathcal{S}_{n+1}$, then $\pi(A)\in \mathcal{S}_n$.
\end{itemize}
An element of $\mathcal{S}_n$ is called {\bf definable in $\mathcal{S}$}.
The structure $\mathcal{S}$ is said {\bf o-minimal} if it satisfies the following condition:
\begin{itemize}
\item [5)] The elements of $\mathcal{S}_1$ are precisely finite unions of points and intervals.
\end{itemize}
\end{definition}

\begin{definition}
A mapping $f\colon A\subset \R^n \to \R^m$ is called {\bf definable in $\mathcal{S}$} if its
graph is an element of $\mathcal{S}_{n+m}$. 
\end{definition}

We say that a set $X\subset \R^n$ is {\bf locally definable in $\mathcal{S}$} if for any $x\in \overline{X}$ there is a neighbourhood $U\subset \R^n$ of $x$ such that $X\cap U$ is definable in $\mathcal{S}$.

Throughout this paper, we fix an o-minimal structure $\mathcal{S}$ on $\R$.

In the sequel, the adjective {\bf definable} denotes definable in $\mathcal{S}$.

We say that a set $X\subset \R^n$ is {\bf locally definable} if for any $x\in \overline{X}$ there is a neighbourhood $U\subset \R^n$ of $x$ such that $X\cap U$ is definable.

\subsection{Smooth mappings on sets}

In this subsection, we recall the definitions of smoothness of mappings defined only in subsets of manifolds. This subsection is closely related to the section with the same title in \cite{Sampaio:2025b}.

\begin{definition}\label{defi:regulaties_one}
Let $(M_1,d_1)$ and $(M_2,d_2)$ be two metric spaces. We say that a map $f\colon M_1\to M_2$ is:
\begin{itemize}
    \item {\bf $\alpha$-H\"older} if there is a constant $C>0$ such that 
    $$
    d_2(f(x),f(y))\leq Cd_1(x,y)^{\alpha} 
    $$
    for all $x,y\in M_1$. In this case, we also say that $f$ is {\bf $C^{0,\alpha}$ smooth}. When $f$ is 1-H\"older, we also say that $f$ is {\bf Lipschitz};
    \item  a {\bf bi-Lipschitz homeomorphism} if $f$ is a bijection such that $f$ and $f^{-1}$ are Lipschitz;
    \item {\bf $\gamma$-log-Lipschitz} if there is a constant $C>0$ such that
    $$
    d_2(f(x),f(y))\leq Cd_1(x,y)|\log (d_1(x,y))|^{\gamma} 
    $$
    for all $x,y\in M_1$ such that $d_1(x,y)\leq 1/2$.
    In this case, we also say that $f$ is {\bf $C^{0,\gamma-log-Lip}$ smooth}. When $f$ is $1$-log-Lipschitz, we simply say that $f$ is {\bf log-Lipschitz} or $f$ is {\bf $C^{0,log-Lip}$ smooth}.
\end{itemize}
When $M_1$ and $M_2$ are two $C^{\infty}$ smooth manifolds, $k\geq 0$ and $\beta\in [0,1]\cup \{\gamma-log-Lip;$ $\gamma\geq 0\}$, we say that a map $f\colon M_1\to M_2$ is $C^{k+1,\beta}$ smooth if its derivative $Df\colon TM_1\to TM_2$ is $C^{k,\beta}$ smooth.
\end{definition}

\begin{definition}\label{defi:regulaties_two}
Let $M_1$ and $M_2$ be two $C^{\infty}$ smooth manifolds. Let $f\colon A_1\to A_2$ be a mapping, where $A_1$ and $A_2$ are subsets of $M_1$ and $M_2$, respectively. We say that $f$ is $C^{k,\alpha}$ (resp. $C^{\infty}$) smooth if for each $p\in A_1$, there are an open neighbourhood $U\subset M_1$ of $p$ and a $C^{k,\alpha}$ (resp. $C^{\infty}$) smooth mapping $F\colon U\to M_2$ such that $F|_{A_1\cap U}=f|_{A_1\cap U}$. We say that $f$ is a $C^{k,\alpha}$ (resp. $C^{\infty}$) diffeomorphism if $f$ is a bijection and, moreover, $f$ and $f^{-1}$ are $C^{k,\alpha}$ (resp. $C^{\infty}$) smooth.
\end{definition}
\begin{definition}\label{defi:regulaties_three}
Let $M_1$ and $M_2$ be two real (resp. complex) analytic manifolds. Let $f\colon A_1\to A_2$ be a mapping, where $A_1$ and $A_2$ are subsets of $M_1$ and $M_2$, respectively. We say that $f$ is real (resp. complex) analytic smooth if for each $p\in A_1$, there are an open neighbourhood $U\subset M_1$ of $p$ and a real (resp. complex) analytic smooth mapping $F\colon U\to M_2$ such that $F|_{A_1\cap U}=f|_{A_1\cap U}$. We say that $f$ is a real (resp. complex) analytic diffeomorphism if $f$ is a bijection and, moreover, $f$ and $f^{-1}$ are real (resp. complex) analytic smooth.
\end{definition}

\subsection{Tangent cones}

\begin{definition}\label{def:tg_cones}
Let $X\subset \R^{n}$ be a set such that $p\in \overline{X}$. We define $C_i(X,p)$, $i=3,4 $, and $5$ as follows:
\begin{enumerate}
 \item $v\in C_3(X,p)$ if there are a sequence of points $\{x_i\}_i\subset X$ tending to $p$ and a sequence of positive real numbers $\{t_i\}_i$ such that 
$$\lim\limits_{i\to \infty} \frac{1}{t_i}(x_i-p)= v.$$
 \item $v\in C_4(X,p)$ if there are a sequence of points $\{x_i\}\subset {\rm Reg}_1X$ tending to $p$ and a sequence of $\{v_i\}_i\subset  \R^{n}$ tending to $v$ such that $v_i\in T_{x_i}X$ for all $i$. Here, ${\rm Reg}_1X$ denotes the set of points $p\in X$ where $X$ is $C^1$ smooth at $p$ (i.e., $X$ is a $C^1$ submanifold around $p$). 
 \item $v\in C_5(X,p)$ if there are sequences of points $\{x_i\}_i, \{y_i\}_i\subset X$ tending to $p$ and a sequence of positive real numbers $\{t_i\}_i$ such that 
$$\lim\limits_{i\to \infty} \frac{1}{t_i}(x_i-y_i)= v.$$
\end{enumerate}
\end{definition}

\begin{lemma}\label{inclusao_cones}
Let $X\subset \R^{n}$ be a definable set such that $p\in \overline{X}$. Then $C_3(X,p)\subset C_4(X,p)\subset C_5(X,p)$.
\end{lemma}
\begin{proof}
First observe that $C_3(X,p)=C_3(Reg_1(X),p)$.
Now given $v \in C_3(X,p)=C_3(Reg_1(X),p)$ we have that there is a curve $\alpha:[0,\epsilon) \rightarrow Reg_1(X)$ such that, among other things, $\alpha'(0)=v$, but since $\alpha(t) \in Reg_1(X)$, we have that $\alpha'(t) \in T_{\alpha(t)}X$. As $\alpha \in C^1$, then $\lim_{t \to 0^+}\alpha'(t)=\alpha'(0)=v$. Take $v \in C_4(X,p)$, since $v_k \in T_{x_k}X$, then there is a curve $\alpha:(-\epsilon,\epsilon) \to X$ with $\alpha (0)=x_k$, such that $v_k=\alpha'(0)$. But this means that there are sequences $(y_{k_i})_i$ and $(z_{k_i})_i$ in $(-\epsilon,\epsilon)$ with $(y_{k_i})_i \le 0 \le (z_{k_i})_i$ and $\lim_{i \to \infty}(y_{k_i})_i=0=\lim_{i \to \infty}(z_{k_i})_i$, so that
$$v_k=\alpha'(0)=\lim_{i \to \infty} \frac{\alpha((z_{k_i})_i)-\alpha((y_{k_i})_i)}{(z_{k_i })_i-(y_{k_i})_i},$$
then for all $v_k$, fixed $i\gg 1$, we obtain $a_k=(z_{k_i})_i$ and $b_k=(y_{k_i})_i$ so that
$$\left\|v_k-\frac{\alpha(a_k)-\alpha(b_k)}{a_k-b_k}\right\|<\frac{1}{2i},$$
then $v \in C_5(X,p)$.   
\end{proof}

\begin{lemma}\label{caracterizacao do C4}
Let $X\subset \R^{n}$ be a pure $d$-dimensional definable set such that $p\in \overline{X}$. Then $C_4(X,p)$ is the union of all $d$-dimensional linear subspaces $T$ such that there is a sequence of points $\{x_i\}_i\subset {\rm Reg}_1X$ tending to $p$ with $\lim\limits_{i\to \infty} T_{x_i}X=T$.
\end{lemma}
\begin{proof}
Let $T$ be a $d$-dimensional linear subspaces such that there is a sequence of points $\{x_j\}_j\subset {\rm Reg}_1X$ tending to $p$ with $\lim\limits_{j\to \infty} T_{x_j}X=T$. We may assume that $\delta(T_{x_j}X,T)<\frac{1}{j}$, where $\delta$ is the usual Grassmannian metric. Let $\{w_1, w_2, \cdots, w_d\}$ be a basis of $T$, which we may assume satisfies $\|w_i\| = 1$ for $1 \le i \le k$.
Since $\delta(T_{x_j}X, T) < \frac{1}{j}$, there exists $v_{i_j} \in T_{x_j}X$ such that $\|v_{i_j} - w_i\| < \frac{1}{i}$. Thus, given $w \in T$, we have $w = a_1w_1 + \cdots + a_dw_d$ for constants $a_1, \cdots, a_d \in \mathbb{R}$, and therefore the sequence of vectors $w_j = a_1v_{1_j} + \cdots + a_nv_{k_j} \in T_{x_j}X$ converges to $w$. Thus, $w\in C_4(X,p)$.

Now take a vector $v \in C_4(X,p)$, then there is a sequence $v_i \to v$, as in the definition, we have that the sequence $(T_{x_i}X)$ converges to a d-dimensional subspace of $\mathbb{R}^n$, say $S$. So short of normalizing the vectors of $(v_i)_i$ and $v$, we have to
$$\frac{1}{i}>\delta(S,T_{x_i}X)\ge sup_{\|w\|=1,w \in T_{x_i}X}\{\|w-\pi_S(w)\|\} \ge \|v_i-\pi_S(v_i)\|,$$ which proves the result.

\end{proof}

\begin{theorem}[Teorema 4.1.1 in \cite{Sampaio:2015}]\label{inv_cones}
Let $X,Y\subset\R^{n}$ be two germs of analytic subsets. If $\varphi\colon(X,p)\to (Y,q)$ is a bi-Lipschitz homeomorphism, then there is a bi-Lipschitz homeomorphism $d\varphi\colon C_3(X,p)\to C_3(Y,q)$ such that $d\varphi(0)=0$.
\end{theorem}

\subsection{The Nash mapping}\label{subsec:nash_map}

Given a pure $d$-dimensional $\mathbb{K}$-analytic set $X$ in $\mathbb{R}^n$, $\mathcal{N}(X)$ is the closure in $X\times Gr_{\mathbb{K}}(d,n)$ of the graph of the {\bf Gauss mapping} $\nu\colon X\setminus Sing(X)\to Gr_{\mathbb{K}}(d,n)$ given by $\nu(x)=T_xX$, where $\mathbb{K}$ is $\C$ or $\R$, and $Gr_{\mathbb{K}}(d,n)$ is the Grassmannian of $d$-dimensional $\mathbb{K}$-linear subspaces in $\mathbb{K}^n$. The set $\mathcal{N}(X)$ is called the {\bf Nash transformation of $X$}, and the projection $\eta\colon \mathcal{N}(X)\to X$ is called the {\bf Nash mapping of $X$}.

Recall that we have a distance defined in $Gr_{\R}(d,n)$ as follows: Given two linear subspaces $L_1$ and $L_2$ in $Gr_{\R}(d,n)$, 
\begin{eqnarray*}
\delta(L_1,L_2)&=& d_H(S_1,S_2)\\
            &=&\max\{\sup \limits_{a\in S_1}d(a,S_2),\sup \limits_{b\in S_2}d(b,S_1)\},
\end{eqnarray*}
where $S_i=L_i\cap \mathbb{S}^{n-1}$, $i=1,2$, and $d(a,S)=\inf\{\|a-x\|;x\in S\}$. For any set $X\subset \R^n$ that is pure $d$-dimensional and locally definable in an o-minimal structure on $\R$, we denote by $\delta_X$ the distance on $X\times Gr_{\R}(d,n)$ given by $\delta_X((x_1,L_1),(x_2,L_2))=\max \{\|x_1-x_2\|,\delta(L_1,L_2)\}$. So, $\mathcal{N}(X)$ is endowed with the distance induced by $\delta_X$.

\subsection{Multiplicity and degree of real sets}
This subsection is closely related to the subsection with the same title in \cite{FernandesJS:2022}.

Let $X\subset \R^{n}$ be a $d$-dimensional real analytic set with $0\in X$ and 
$$
X_{\C}= V(\mathcal{I}_{\R}(X,0)),
$$
where $\mathcal{I}_{\R}(X,0)$ is the ideal in $\mathbb{C}\{z_1,\ldots,z_n\}$ generated by the complexifications of all germs of real analytic functions that vanish on the germ $(X,0)$. We know that $X_{\C}$ is a germ of a complex analytic set and $\dim_{\C}X_{\C}=\dim_{\R}X$. Then, for a linear projection $\pi:\C^{n}\to\C^d$ such that $\pi^{-1}(0)\cap C(X_{\C},0) =\{0\}$, there exists an open neighbourhood $U\subset \C^n$ of $0$ such that $\# (\pi^{-1}(x)\cap (X_{\C}\cap U))$ is constant for a generic point $x\in \pi(U)\subset\C^d$. This number is the multiplicity of $X_{\C}$ at the origin and is denoted by $m(X_{\C},0)$.

\begin{definition}
With the above notation, we define the multiplicity of $X$ at the origin by $m(X,0):=m(X_{\C},0)$.
\end{definition}

More about the multiplicity of real analytic sets can be learnt in \cite{Sampaio:2022}.

\begin{remark}\label{transversal-cones}
It follows from the works in \cite{Sampaio:2022} and \cite{FernandesJS:2022} that if $X\subset \R^n$ is a pure $d$-dimensional real analytic set and $\pi\colon \R^n\to \R^d$ is a projection such that $\pi^{-1}(0)\cap C(X,0)=\{0\}$, then there are an open neighbourhood $U\subset \R^n$ of $0$ and a subanalytic set $\sigma\subset \R^d$ such that $\dim_{\R}\sigma <d$ and $m(X,0)\equiv \# (\pi^{-1}(t)\cap U)(\mod \, 2)$ for all small enough $t\in \R^d\setminus \sigma$, where $\# (\pi^{-1}(t)\cap U)$ denotes the cardinality of $\pi^{-1}(t)\cap U$.
\end{remark}

\section{Smoothness of definable sets}\label{sec:reg_definable}

In summary, in this Section, we prove the following.

\begin{theorem}\label{thm:main_thm_complete}
Let $X\subset \R^n$ be a $d$-dimensional definable set in an o-minimal structure. Then the following items are equivalent.
\begin{enumerate}
 \item  $X$ is locally closed and LNE at any $p\in X$, and $C_3(X,p)$ is a $d$-dimensional linear subspace for any $p\in X$ and depends continuously on $p$;
 \item For each $p\in X$, $X$ is Lipschitz regular at $p$ and $C_4(X,p)$ is a $d$-dimensional linear subspace;
 \item $X$ is a topological manifold and for each $p\in X$, $X$ is LNE at $p$ and $C_4(X,p)=C_3(X,p)$;
 \item $X$ is a topological manifold and $C_5(X,p)$ is a $d$-dimensional linear subspace set for any $p\in X$;
 \item $X$ is $C^1$ smooth.
\end{enumerate}
\end{theorem}

\subsection{Smoothness and the third tangent cone}
\begin{theorem}\label{thm:main_thm_one}
Let $X\subset \R^n$ be a $d$-dimensional locally closed definable set in an o-minimal structure. Then $X$ is $C^1$ smooth if and only if $X$ is LNE at any $p\in X$, $C_3(X,p)$ is a $d$-dimensional linear subspace for any $p\in X$, and it depends continuously on $p$.
\end{theorem}
\begin{proof}
Clearly, if $X$ is $C^1$ smooth, then $X$ is LNE at any $p\in X$, $C_3(X,p)$ is a $d$-dimensional linear subspace for any $p\in X$, and it depends continuously on $p$.

Reciprocally, we assume that $X$ is LNE at any $p\in X$, $C_3(X,p)$ is a $d$-dimensional linear subspace for any $p\in X$, and depends continuously on $p$.
Fix $p\in X$. We may assume that $p=0$.

Let $\Pi\colon \R^n\to C_3(X,0)$ be the orthogonal projection onto $C_3(X,0)$.
We choose linear coordinates $(x,y)$ in $\R^n$ such that $C_3(X,0)=\{(x,y)\in \R^n;y=0\}\cong \R^d$. With this identification, $\pi(x,y)=x$. Let $L={\rm Ker}(\Pi)=\{(x,y)\in \R^n; x=0\}$.

By hypotheses, $C_3(X,p)$ is a $d$-dimensional linear subspace for any $q\in X$, and it depends continuously on $q$. Thus, we can choose a small enough open neighbourhood $U$ of $0$ such that $C_3(X,q)\cap L=\{0\}$ for all $q\in \overline{U}\cap X$. In particular, $\Pi|_{C_3(X,q)}\colon C_3(X,q)\to C_3(X,0)$ is a linear isomorphism.

Since $X$ is a locally closed set, we can shrink $U$, if necessary, so that $X\cap \overline{U}$ is a compact subset.

By shrinking $U$ again, if necessary, we may assume that $(X\cap U)\cap \Pi^{-1}(\Pi(q))$ is a finite subset for all $q\in X\cap U$.
Indeed, if this was not possible, there would be $q\in X\cap U$ and a sequence $\{q_j\}_j\subset (X\setminus \{\tilde q\})\cap \Pi^{-1}(\Pi(q))$ such that $\lim q_j=\tilde q\in X\cap \overline{U}$. Taking a subsequence, if necessary, we may assume that $\lim\frac{q_j-\tilde q}{\|q_j-\tilde q\|}=v$. Then $v\in C_3(X,\tilde q)\cap L$, which is a contradiction, since $C_3(X,\tilde q)\cap L=\{0\}$ and $v\not=0$.

From above, we obtain, in particular, that $(X\cap U)\cap L$ is a finite subset and, thus by shrinking $U$ again, if necessary, we may assume that $(X\cap \overline{U})\cap L=\{0\}$.

\begin{claim}\label{claim:open_map}
$\pi=\Pi|_A$ is an open mapping, where $A=X\cap U$. 
\end{claim}
\begin{proof}[Proof of Claim \ref{claim:open_map}]
This claim is already proved in  \cite[Lemma 4.1]{KurdykaGN:2018}, but here we present a slightly different proof following the idea of \cite[Claim 3 of Theorem 3.1]{Sampaio:2021b}. 
In anyway, our proof still shares its structure with the proof of Lemma 3.3 in \cite{GhomiH:2014} and Lemma 4.1 in \cite{KurdykaGN:2018}.

Fix $q \in A$. Since $X$ is locally closed, there is a $r > 0$ small enough such that $V:=\overline{B_r^n(q)}\cap X \subset A$ is a compact subset, and the topological boundary $\partial V$ of $V$ in $X$ lies on $\partial B_r^n(q)$, 
which implies that $q\not\in\partial V.$ Moreover, for a small enough $r>0$, we can assume that
$$
\pi(q)\not\in \pi(\partial V)
$$
otherwise, there exists a sequence of positive numbers $\{r_k\}$ tending to $0$ such that for each $k$, there is a point $q_k \in X \cap \partial\overline{ B_{r_k}^n(q)}$ with $\pi(q_k) = \pi(q)$. So, extracting a subsequence if necessary, we can assume that $\lim\limits_{k \to \infty} \frac{q_k-q}{\|q_k-q\|}=v\not=0$. This implies that $v \in L \cap C_3(X,q)$, which is a contradiction, since $C_3(X,q)\cap L=\{0\}$. 
Since $\pi (\partial V)$ is a compact set, there is $s > 0$ such that 
\begin{equation}\label{eq:pidv}
\overline{B_s^{d}(\pi(q))}\cap  \pi(\partial V) = \emptyset.
\end{equation} 
It is enough to show that $\pi(q)$ is an interior point of $\pi(V)$ in $P$.

Thus, suppose by contradiction that $\pi(q)$ is not an interior point of $\pi(V)$. Then ${B_{\delta}^{d}(\pi(q))}\not\subset \pi(V)$, for any $\delta>0$. In particular, there is a point
$$
x\in {B_{s/2}^{d}(\pi(q))}\setminus\pi(V).
$$ 
Since $x\not\in\pi(V)$ and $\pi(V)$ is compact, for $t=dist(x,\pi(V))$, we have that $\overline{B_{t}^{d}(x)}\subset C_3(X,0)$ intersects $\pi(V)$ while $B_{t}^{d}(x)\cap \pi(V)=\emptyset$. Moreover, since $q\in V$, we have 
$t\leq\|x-\pi(q)\|< s/2$, and if $x'\in \overline{B_{t}^{d}(x)}$ then
$$
\|x'-\pi(q)\|\leq \|x'-x\|+\|x'-\pi(q)\|\leq t+s/2< s,
$$
which  yields
$$
\overline{B_{t}^{d}(x)}\subset B_s^{d}(\pi(q)).
$$
Thus $\overline{B_{t}^{d}(x)}\cap\pi(\partial V)=\emptyset$ by (\ref*{eq:pidv}).
Take $\tilde x \in \overline{B_{t}^{d}(x)} \cap \pi(V)$ and $\tilde p \in \pi^{-1}(\tilde x) \cap V$. Note that $\tilde p \not \in \partial V$, so $\tilde p$ is an interior point of $V$, and hence $C_3(V,\tilde p) = C_3(X,\tilde p)$.

Since $B_{t}^{d}(x) \cap \pi(V)  = \emptyset$, no point of $V$ is contained in the cylinder $C: = \pi^{-1} (B_{t}^{d}(x))$. This implies that $ \ell \subset T_{\tilde p} \partial C$ for each line $\ell\subset C_3(X,\tilde p)$ passing through the origin. In fact, let $\ell\subset C_3(X,\tilde p)$ be a line passing through the origin, then we have two arcs $\gamma_1,\gamma_2\colon[0,\varepsilon)\to V$ such that $\gamma_i(\tau)-\tilde p=(-1)^i\tau u+o(\tau)$ for $i=1,2$, where $u\in \R^n\setminus \{0\}$ satisfies $\ell=\{\tau u;\tau\in \R\}$. Thus, suppose that $ \ell \not\subset T_{\tilde p} \partial C$, then the line $\tilde p+\ell= \{\tilde p+\tau u;\tau\in \R\}$ and $C$ intersect transversally at $\tilde p$, this forces the image of $\gamma_1$ or $\gamma_2$ to intersect $C$, which is a contradiction, since $V\cap C=\emptyset$. Therefore, $ \ell \subset T_{\tilde p} \partial C$. Since $C_3(X,\tilde p)$ is linear, we obtain that $ C_3(X,\tilde p) \subset T_{\tilde p} \partial C$.

Since $L \subset T_{\tilde p} \partial C$ and $C_3(X,\tilde p)\oplus L=\R^n$, we obtain that $T_{\tilde p} \partial C=\R^n$, which is a contradiction, since $T_{\tilde p} \partial C$ is a $(n-1)$-dimensional linear subspace.

Therefore, $\pi(q)$ is an interior point of $\pi(V)$ and this finishes the proof of Claim \ref{claim:open_map}.

\end{proof}

\begin{claim}\label{claim:injective_map}
$\pi$ is an injective mapping.
\end{claim}
\begin{proof}[Proof of Claim \ref{claim:injective_map}]
Let $k=\max\{\#(\pi^{-1}(x));x\in \pi(A)\}$ and $E=\{x\in \pi(A); k=\#(\pi^{-1}(x))\}$. We have that $E$ is an open subset of $\pi(A)$ and $\pi\colon \pi^{-1}(E)\to E$ is a cover mapping. Since $\#(\pi^{-1}(0))=1$, we have that $k>1$ if and only if $\pi(A)\setminus E\not=\emptyset$.

Assume that $k>1$ and let $x\in E$. Let $z\in \pi(A)\setminus E$ such that $r={\rm dist}(x,\pi(A)\setminus E)=\|x-z\|$. Then, $B_r(x)\subset E$ and there are $k$ Lipschitz mappings $f_1,...,f_k\colon B_r(x)\to \R^{n-d}$ such that $Graph(f_i)\subset A$, $i=1,...,k$. We have that each $f_i$ extends to a Lipschitz mapping $\tilde f_i\colon \overline{B_r(x)}\to \R^{n-d}$. Since $z\not\in E$, then $\#(\pi^{-1}(z))<k$. Thus, there are different $i_1,i_2\in\{1,...,k\}$ such that $\tilde f_{i_1}(z)=\tilde f_{i_2}(z)$. After reordering the indices, if necessary, we assume that $i_1=1$ and $i_2=2$.

Since $Graph(\Tilde{f_i})$ is bi-Lipschitz homeomorphic to the ball $\overline{B_r(x)}$ (which has a half-space as its tangent cone), since $C_3(X,p)$ is a linear subspace at every point $p \in X$, we have that the tangent cone of $Graph(\Tilde{f_i})$ at $\tilde{z}=(z,f(z))$ is also a half-space and is given by the graph of the function $C_3(\overline{B_r(x)},z) \ni v \mapsto (v,df(x)v)$, where $df$ represents the lateral derivative. Take two points $(y,f_1(y)) \in Graph(\Tilde{f_1})$ and $(y,f_2(y)) \in Graph(\Tilde{f_2})$ and curves $(\gamma_i(t)) \subset Graph(\Tilde{f_i})$, with $\gamma_i(0)=\tilde{z}$ for $i=1,2$. Since $C_3(Graph(\Tilde{f_1}),\tilde{z})=C_3(Graph(\Tilde{f_2}),\tilde{z})$, such curves can be taken so as to have the same tangent vector. But this is a contradiction, since $\frac{|\gamma_1(t)-\gamma_2(t)|}{t} \ge C>0$.

Therefore, $k=1$ and $\pi $ is an injective mapping.

\end{proof}

Let $\phi\colon (\R^{d},0)\to \R^{n-d}$ be the mapping-germ such that $\pi(x,\phi(x))=x$ for all small enough $x\in \R^d$.
\begin{claim}\label{claim:graph_c_one}
$\phi $ is a mapping of class $C^1$.
\end{claim}
\begin{proof}[Proof of Claim \ref{claim:graph_c_one}]
This claim follows from \cite[Lemma 4.2]{KurdykaGN:2018}, but for completeness, we present that proof here.

Let us begin by showing that $\phi$ is continuous. Let $U$ be the domain of $\phi$, and let $a \in U$. Take a sequence $\{x_k\}_k \to a$, then by Claim \ref{claim:open_map}, there exists a neighborhood $V$ of $(a,\phi(a))$ in $X$ such that $\pi|_V$ is open, where $\pi:\mathbb{R}^n \rightarrow \mathbb{R}^k$. Since $X$ is locally closed, we may assume, possibly by shrinking $V$, that $\overline{V} \subset X = \Gamma(\phi)$ is compact. As $a \in \pi(V)$, we have that for $k \gg 1$, $x_k \in \pi(V)$, so $\{(x_k,\phi(x_k))\}_k \subset \overline{V}$. Therefore, possibly after passing to a subsequence, we may assume that this sequence converges, say, to $(a,b)$. Thus, since $\overline{V} \subset \Gamma_{\phi}$, it follows that $\phi(a) = b = \lim_{k \to \infty} \phi(x_k)$. This shows that $\phi$ is continuous.

Now we show that $\phi$ is $C^1$.

Let $\sigma\colon \mathcal{L}(\mathbb{R}^k,\mathbb{R}^{n-k}) \ni \phi \mapsto \Gamma(\phi) \in Gr_{\R}(k,n)$, that is, the map that associates a linear map $\phi$ to its graph. But since $\sigma$ is a homeomorphism onto its image, and since $C_3(X,p)$ is not orthogonal to $\mathbb{R}^k$, we have that $C_3(X,p) \in \sigma(\mathcal{L}(\mathbb{R}^k,\mathbb{R}^{n-k}))$. Let $L(x)=\sigma^{-1}(C_3(X,(x,\phi(x))))$. But by construction, we have that $d\phi(x)=L(x)$, which is a composition of continuous functions, therefore $\phi \in C^1$.

\end{proof}
 Therefore, $X$ is locally the graph of a $C^1$ smooth mapping, and thus $X$ is $C^1$ smooth.
\end{proof}

\subsubsection{Examples}
Let us present and analyse some examples in order to show the sharpness of Theorem \ref{thm:main_thm_one}.

After proving Theorem \ref{thm:main_thm_one}, the following question naturally arises: {\it is the above theorem still valid if, instead of requiring $C_3(X,p)$ as in the theorem, we require that $C_4(X,p)$ or $C_5(X,p)$ be a $d$-dimensional linear subspace and that $X$ is LNE at $p$, for all $p \in X$?} 

It is not difficult to note that under these conditions, $C_4(X,p)$ and $ C_5(X,p)$ vary continuously. The example below shows that the above question has a negative answer.

\begin{example}\label{thm:main_exem_one}
Consider the set $X = \{(x,y) \in \mathbb{R}^2; y^2 - x^3 \le 0\}$, which is definable and satisfies $C_4(X,p) = \mathbb{R}^2 = C_5(X,p)$ for all $p \in X$. Moreover, $X$ is LNE, but it is not $C^1$-smooth.    
\end{example}

The condition ``$X$ is LNE at $p$'' cannot be removed.

\begin{example}
For $X=\{(x,y)\in \C^2;x^2-y^3=0\}$, we have that $C_3(X,p)$ is a complex line for all $p\in X$, depending continuously on $p$, but $X$ is not $C^1$ smooth.
\end{example}

The condition ``$C_3(X,p)$ depends continuously on $p$'' cannot be removed.

\begin{example}\label{exam:flat_tg_cone_not_enough}
Let $f\colon \R^2\to \R$ be the function given by 
$$
f(x,y)=\left\{\begin{array}{ll}
                \frac{y^6}{x^2+y^4},& \mbox{ if }x^2+y^4\not=0,\\
               0,&\mbox{ if }x^2+y^4=0.
              \end{array}\right.
$$
Note that $f$ is a Lipschitz function but not a $C^1$ function. Indeed, for $x^2+y^4\not=0$, the partial derivatives $
\frac{\partial f}{\partial x}(x,y)=\frac{-2xy^6}{(x^2+y^4)^2}$ and $
\frac{\partial f}{\partial x}(x,y)=\frac{6y^5}{x^2+y^4}-\frac{4y^9}{(x^2+y^4)^2}$ are bounded around $(0,0)$ and $\frac{\partial f}{\partial x}$ cannot be extended as a continuous function at $(0,0)$.
Thus, $X=Graph(f)$ is Lipschitz regular at $0$. Since $X\setminus\{0\}$ is smooth, $C_3(X,p)$ is plane for all $p\in X\setminus\{0\}$. Moreover, $P=\{(x,y,z)\in \R^3;z=0\}\subset C(X,0)$. Since $X$ is Lipschitz regular at $0$, then $C(X,0)$ is lipeomorphic to $P$, and thus $C(X,0)=P$. Therefore, $X$ is Lipschitz regular at any $p\in X$, $C_3(X,p)$ is plane for all $p\in X$, but $X$ is not $C^1$-smooth at $0$.
\end{example}

We can obtain examples with densities as close to 1 as we want.
\begin{example}\label{exam:small_density_not_enough}
For each $\varepsilon>0$, let $X_{\varepsilon}=Graph(\varepsilon f)$, where $f$ is the function defined in Example \ref{exam:flat_tg_cone_not_enough}. Since $f$ is Lipschitz, the density of $X_ {\varepsilon}$ at $0$, $\theta(X_ {\varepsilon},0)$, goes to 1 when $\varepsilon \to 0^+$ and, in particular, $\theta(X_ {\varepsilon},0)<\frac{3}{2}$ for a small enough $\varepsilon>0$. However, $X_{\varepsilon}$ is not $C^1$ smooth at $0$ and $C_3(X_{\varepsilon},p)$ is plane for all $p\in X_{\varepsilon}$.
\end{example}

We can obtain even an algebraic example.

\begin{example}
 Let $X=\{(x,y,z)\in \R^3;z^{15}=x^9y^7+x^7y^9\}$. We have that for each $p\in X$, $X$ is Lipschitz regular at $p$ and $C_3(X,p)$ is a plane. However, $X$ is not $C^1$ smooth at $0$.
\end{example}

\subsection{Smoothness and the forth tangent cone}
\begin{theorem}\label{thm:main_thm}
Let $X\subset \R^n$ be a $d$-dimensional definable set in an o-minimal structure. Then $X$ is $C^1$ smooth if and only if $X$ is Lipschitz regular at any $p\in X$, and $C_4(X,p)$ is a $d$-dimensional linear subspace for any $p\in X$.
\end{theorem}
\begin{proof}

It is clear that if $X$ is $C^1$ smooth, then $X$ is Lipschitz regular at any $p\in X$, and $C_4(X,p)$ is a $d$-dimensional linear subspace for any $p\in X$.

Assume that $C_4(X,p)$ is a $d$-dimensional linear subspace for any $p\in X$. Fix $p\in X$. By hypotheses, $X$ is Lipschitz regular at $p$.
By the bi-Lipschitz invariance of the tangent cones, there is a lipeomorphism $\phi\colon \R^d\to C_3(X,p)$ such that $\phi(0)=0$. In particular, the mapping $\phi\colon \R^d\to C_4(X,p)$ is a continuous and injective mapping, and thus an open mapping. This shows that $C_3(X,p)$ is an open subset of $C_4(X,p)$. Since $C_3(X,p)$ is a closed subset of $C_4(X,p)$ and $C_4(X,p)$ is a connected set, we have that $C_3(X,p)=C_4(X,p)$.

Therefore, $C_3(X,p)$ is a $d$-dimensional linear subspace for any $x\in X$.

\begin{claim}\label{claim:cont_c4}
$C_4(X,p)$ depends continuously on $p$.
\end{claim}
\begin{proof}[Proof of Claim \ref{claim:cont_c4}]
Suppose that the claim is not true. Then, since the Grassmannian is compact, there are $p\in X$ and sequences $(x_k)_k,(y_k)_k \subset X$, both converging to $p$ such that $\lim C_4(X,x_n)=S \ne T=\lim C_4(X,y_n)$. By Proposition $\ref{caracterizacao do C4}$, there exist $x'_n,y'_n \in Reg_1(X)$, both converging to $p$, such that $d(C_4(X,x_n),T{x'_n}X)<\frac{1}{n}$ and $d(C_4(X,y_n),T{y'_n}X)<\frac{1}{n}$. Then $T{x'_n}X \to S$ and $T{y'_n}X \to T$. However, $T,S \subset C_4(X,p)$, and since $T$, $S$ and $C_4(X,p)$ are subspaces of dimension $d$, it follows that $T=C_4(X,p)=S$, which is a contradiction.
\end{proof}

Therefore, $C_3(X,p)$ is a $d$-dimensional linear subspace for any $p\in X$, and depends continuously on $p$ and, moreover, $X$ is LNE at any $p\in X$. By Theorem \ref{thm:main_thm_one}, $X$ is $C^1$ smooth.
\end{proof}

\subsubsection{Examples}
Let us present and analyse some examples in order to show the sharpness of Theorem \ref{thm:main_thm}.

The condition that $X$ is Lipschitz regular at any $p\in X$ cannot be removed, even if we require that $X$ is a topological manifold.

\begin{example}\label{exam:real_cusp}
For $X=\{(x,y)\in \R^2;x^2-y^3=0\}$, we have that $X$ is a topological manifold and $C_4(X,p)$ is a real line for all $p\in X$. But $C_3(X,0)$ is not a real line.
\end{example}
We can even obtain an example that is a complex algebraic set.
\begin{example}\label{exam:complex_cusp}
For $X=\{(x,y)\in \C^2;x^2-y^3=0\}$, we have that $X$ is a topological manifold and $C_4(X,p)=C_3(X,p)$ is a complex line for all $p\in X$, but $X$ is not $C^1$ smooth.
\end{example}

The condition that $C_4(X,p)$ is a $d$-dimensional linear subspace for any $p\in X$ cannot be removed. 

\begin{example}
For $X=\{(x,y)\in \R^2;y=|x|\}$, we have that $X$ is Lipschitz regular at any $p\in X$, but $X$ is not $C^1$ smooth.
\end{example}

The condition that $C_4(X,p)$ is a linear subspace for all $p\in X$ cannot be weakened to ``$C_3(X,p)$ being a plane for all $p\in X$''. For instance, see Example \ref{exam:flat_tg_cone_not_enough}.

\subsection{Smoothness and the third and fourth tangent cones}

\begin{theorem}\label{thm:c3=c4_regularity}
     Let $X \subset \mathbb{R}^n$ be a definable  set in an o-minimal structure. Then $X$ is $C^1$ smooth if and only if $X$ is a topological manifold that is LNE at any $p\in X$ and $C_3(X,p) = C_4(X,p)$ for all $p \in X$.
\end{theorem}

\begin{proof}
It is clear that if $X$ is $C^1$ smooth, then $X$ is a topological manifold that is LNE at any $p\in X$ and $C_3(X,p) = C_4(X,p)$ for all $p \in X$.

Reciprocally, assume that $X$ is a topological manifold that is LNE at any $p\in X$ and $C_3(X,p) = C_4(X,p)$ for all $p \in X$.

Let us show that $C_4(X,p)$ is a $d-$dimensional linear subspace for every $p \in X$. If $p \in Reg_1(X)$, there is nothing to do.
By Proposition V.11 in \cite{Parusinski:2023}, there exists a definable subset $Y \subset X$, $\dim Y \le d - 2$, such that any $p \in X \setminus (Reg_1(X) \cup Y)$ satisfies the following property: There are $m \in \mathbb{N}$ and a neighbourhood $V$ of $p$ such that
\begin{equation*}
   Reg_1(X) \cap V = X_1 \sqcup \dots \sqcup X_m, 
\end{equation*}
with $X_i$ connected, and such that for every $i = 1, \dots, m$, the closure of each $X_i$ in $V$ is a $C^1$ manifold with boundary $V \cap (X \setminus Reg_1(X))$.

Thus, if $p \in X\setminus(Reg_1(X)\cup Y)$, then, since closure $X_i$ is a $C^1$ manifold with boundary and $p \in \partial X_i$, we have that $C_4(X_i,p)$ is a single $d-$dimensional linear subspace, for $1\le i \le m$. But since $X$ is a topological manifold, we have only two components, $X_1$ and $X_2$. Since $X_i$ is a $C^1$ manifold with boundary and $p \in \partial X_i$, it follows that $C_3(X_i,p)$ is a $d$-dimensional half-space. Therefore, if $C_3(X_1,p) \cup C_3(X_2,p)$ were not a d-dimensional linear subspace, then by the hypothesis that $C_3(X,p)=C_4(X,p)$, we would have that $C_4(X,p)=C_3(X_1,p) \cup C_3(X_2,p)$ would be two distinct $d$-dimensional linear subspaces, which is a contradiction since it is a union of two d-dimensional half-spaces.

Now let us show that the same holds at points $p \in Y$.
\begin{claim}\label{claim:c3c4_linear}
$C_4(X,p)$ is a $d-$dimensional linear subspace for all $p \in Y$.
\end{claim}
\begin{proof}[Proof of Claim \ref{claim:c3c4_linear}]
    Since X is definable, then $C_4(X,x_0)=C_3(X,x_0)$ is definable, therefore it is simply a finite union of d-dimensional subspaces.
    Suppose there exists $x_0 \in Y$ such that $C_4(X,x_0)=C_3(X,x_0)$ is at least two $d-$dimensional linear subspaces. Since $\dim(Y)\le d-2$, then $X\setminus Y$ is path-connected, that is, given two points $p,q \in X\setminus Y$, there exists a path $\alpha:[0,1]\rightarrow X$ such that $\alpha([0,1])\subset X\setminus Y$, consequently we have that $C_4(X,\alpha(t))$ is a $d-$dimensional linear subspace, for every $t \in [0,1]$. Since we are assuming that $C_4(X,x_0)$ is at least two $d-$dimensional linear subspaces, say $T_1$ and $T_2$, we can assume that distance, $\delta(T_1,T_2)>\epsilon$. By the definition of $C_4(X,0)$, there exist sequences $(x_k)_k$ and $(y_k)_k$ of regular points converging to $x_0$, such that $T_1=\lim_{x_k \to x_0}T_{x_k}X$ and $T_2=\lim_{y_k \to x_0}T_{y_k}X$, but since $(x_k)_k$ and $(y_k)_k$ are sequences of regular points,of dimension $d$, then $(x_k)_k,(y_k)_k\subset X\setminus Y$, so there exists a path $\alpha_n:[0,1] \to X\setminus Y$ such that $\alpha_n$ connects the points $x_n$ and $y_n$. But as $\delta(T_1,T_2)>\epsilon$, there exists $N_0 \in \mathbb{N}$, such that $\delta(T_1,T_{x_m}X)<\frac{\epsilon}{2025}$ and $\delta(T_2,T_{y_m}X)<\frac{\epsilon}{2025}$, for $m\ge N_0$. Then there exists a sequence $(\alpha_n(t_n))_n$, where $\delta(T_1,T_{\alpha_n(t_n)}X)>\frac{\epsilon}{2025}$ and $\delta(T_2,T_{\alpha_n(t_n)}X)>\frac{\epsilon}{2025}$, so up to taking a subsequence, we can assume that $T_{\alpha_n(t_n)}X \to T_3$ which is different from $T_1$ and $T_2$, but repeating this argument $m+1$ times, where $m=\#C_4(X,x_0)$, we reach a contradiction.
\end{proof}
Thus $C_3(X,p)$ is a $d-$dimensional linear subspace for all $p \in Y$, and it depends continuously on $p$ (see Claim \ref{claim:cont_c4}). Since $ X$ is LNE at any $p\in X$, by Theorem \ref{thm:main_thm_one}, $X$ is $C^1$ smooth.
\end{proof}

\begin{corollary}
Let $X\subset \R^n$ be a $d$-dimensional definable set in an o-minimal structure. Then $X$ is smooth if and only if $X$ is Lipschitz regular at any $p\in X$ and $C_4(X,p)=C_3(X,p)$ for any $p\in X$.
\end{corollary}

We obtain the following generalization for Ghomi-Howard's theorem (Theorem \ref{thm:GhomiH}).
\begin{corollary}\label{cor:gen_thm:GhomiH}
Let $X\subset \R^n$ be an $(n-1)$-dimensional definable set in an o-minimal structure. Then, $X$ is $C^1$ smooth if and only if $X$ is a topological manifold and $C_3(X,p)=C_4(X,p)$ for all $p\in X$.
\end{corollary}

\begin{proof}
Since $X$ is a topological manifold, it follows from the proof of Theorem \ref{thm:c3=c4_regularity} that $C_3(X,p)=C_4(X,p)$ is a $d-$dimensional linear subspace for all $p \in X$. Therefore, we have that $C_3(X,p)$ varies continuously on $p$, since $C_4(X,p)$ is a $d-$dimensional linear subspace. It follows from Theorem \ref{thm:GhomiH} that $X$ is $C^1$ smooth.
\end{proof}

Another consequence that follows from the proof of Theorem \ref{thm:c3=c4_regularity} is the following:

\begin{corollary}
    Let $X \subset \mathbb{R}^n$ be a definable set in an o-minimal structure such that $C_4(X,p)$ is a finite union of $d-$dimensional linear subspaces. If $C_3(X,p)$ is $d-$dimensional linear subspace, then $C_4(X,p)$ is $d-$dimensional linear subspace.
\end{corollary}

\subsubsection{An application on the study of the Nash mapping}\label{sec:nash_map}
As a consequence of Theorem \ref{thm:main_thm} and \cite[Remark 3.4]{Sampaio:2025b}, we obtain the following result.
\begin{corollary}\label{cor:real_analytic_ck_smooth_hoelder}
Let $X\subset \R^n$ be a locally closed set that is pure $d$-dimensional and locally definable in an o-minimal structure on $\R$, and let $k$ be a nonnegative integer and $\alpha\in [0,1]$. Then $X$ is $C^{k+1,\alpha}$ smooth if, and only if, the mapping $\eta\colon\mathcal{N}(X)\to X$ is a homeomorphism such that $\eta^{-1}$ is $C^{k,\alpha}$ smooth, $X$ is LNE at each $p\in X$ and $C_3(X,p) = C_4(X,p)$ for all $p \in X$.
\end{corollary}

\subsubsection{Examples}
Let us present and analyse some examples in order to show the sharpness of Theorem \ref{thm:c3=c4_regularity}.

The condition that $X$ is LNE at any $p\in X$ cannot be removed. For instance, see Examples \ref{exam:real_cusp} and \ref{exam:complex_cusp}.

The condition that $X$ is a topological manifold cannot be removed.

\begin{example}
For $X=\{(x,y)\in \R^2;xy=0\}$, we have that $X$ is LNE at any $p\in X$ and $C_4(X,p)=C_3(X,p)$ for all $p\in X$, but $X$ is not $C^1$ smooth.
\end{example}

The condition that $C_4(X,p)=C_3(X,p)$ for all $p\in X$ cannot be removed. 

\begin{example}
For $X=\{(x,y)\in \R^2;y=|x|\}$, we have that $X$ is a topological manifold that is LNE at any $p\in X$, but $X$ is not $C^1$ smooth.
\end{example}

\subsection{Smoothness and the fifth tangent cone}
\begin{proposition}\label{prop:dim_c5}
Let $X \subset \mathbb{R}^n$ be a closed, $d$-dimensional set definable in an o-minimal structure. If $dim(C_5(X,p))=d$ and $\pi\colon \R^n\to \R^d$ is a linear projection such that $\pi^{-1}(0)\cap C_5(X,p)={0}$, then $X$ is, in a neighborhood of $p$, the graph of a Lipschitz function defined on $\pi(U)$, for some open set $U$ of $X$ containing $p$.
\end{proposition}

\begin{proof}
Note that there exists a neighborhood $U$ of $0$ in $X$ such that $\pi:X\cap U \rightarrow P$ is injective. Otherwise, given any neighborhood of $0$, there would exist $x_k,y_k \in U$ such that $\pi(x_k)=\pi(y_k)$ and, therefore, $\frac{x_k-y_k}{\|x_k-y_k\|}$ lies in the kernel of $\pi$. But, up to a subsequence, we can assume that $v=\lim_{k\to \infty}\frac{x_k-y_k}{\|x_k-y_k\|}$ exists, hence $v \in C_5(X,0)$ and also lies in the kernel of $\pi$. Thus, on one hand, $v=0$, and on the other hand, it has norm $1$, which is a contradiction. Therefore, such a neighborhood exists. Since $\pi:U\cap X \rightarrow \mathbb{R}^d$ is injective, we can define a function $f$ on $\pi(U\cap X)$ such that the graph of $f$ is $X \cap U$. It remains to show that $f$ is Lipschitz.

Suppose, for contradiction, that $f$ is not Lipschitz. Then for every given $k \in \mathbb{N}$, there exist sequences $(x_k)_k,(y_k)_k \subset \pi(X\cap U)$ such that
$$\frac{\|f(x_k)-f(y_k)\|}{\|(x_k-y_k)\|} \ge k,$$
but since $\frac{((x_k,f(x_k))-(y_k,f(y_k)))}{\|((x_k,f(x_k))-(y_k,f(y_k)))\|}$ has norm 1, it converges, up to a subsequence. However,
$$\frac{\|x_k-y_k\|}{\|((x_k,f(x_k))-(y_k,f(y_k)))\|}\le \frac{\|x_k-y_k\|}{\|f(x_k)-f(y_k)\|} \to 0,$$
so the sequence $\frac{((x_k,f(x_k))-(y_k,f(y_k)))}{\|((x_k,f(x_k))-(y_k,f(y_k)))\|}$ converges to a vector that lies in $\pi^{-1}(0) \cap C_5(X,0)=\{0\}$, but on the other hand, it has norm $1$, which is a contradiction. Therefore, $f$ is Lipschitz. 
\end{proof}

The converse of the above proposition does not hold.
\begin{example}\label{Graf_mod}
Let $X=\{(x,y)\in R^2; y=|x|\}$. We have that $X$ is the graph of a Lipschitz function, but $C_{5}(X,0)=\{(x,y)\in R^2;-|x| \leq y\leq |x| \}$.
\end{example}

Thus, we have our first result in this subsection regarding the regularity of sets.
\begin{theorem}\label{thm:dim_c5}
Let $X\subset \R^n$ be a locally closed $d$-dimensional definable set in an o-minimal structure. Then, $X$ is $C^1$ smooth if and only if $C_5(X,p)$ is a $d$-dimensional linear subspace and the restriction to $X$ of the orthogonal projection $\pi\colon \R^n\to C_5(X,p)$ is, around $p$, an open mapping for any $p\in X$.
\end{theorem}

\begin{proof}
Since $C_5(X,p)$ is a $d$-dimensional linear subspace, then, up to a linear change of coordinates, we may assume that $C_5(X,p)=\{(x,y)\in \mathbb{R}^d\times \mathbb{R}^{n-d};y=0\}$. Let $\pi:\mathbb{R}^n \rightarrow C_5(X,p)$ be the canonical projection. We have that $\pi^{-1}(0) \cap C_5(X,p)=\{0\}$. But by the same argument used in Proposition~\ref{prop:dim_c5}, we have that $\pi|_{X \cap U}:X \cap U \rightarrow \pi(X\cap U)$ is bi-Lipschitz, and since the restriction to $X$ of the orthogonal projection $\pi\colon \R^n\to C_{5}(X,p)$ is an open mapping for any $p\in X$, then $\pi(X\cap U) \subset P$ is open. Hence, $X$ is Lipschitz regular at $p$.
Moreover, since $X$ is $d$-dimensional, we have that $C_4(X,p)$ is a union of $d$-dimensional linear subspaces, but since $C_4(X,p) \subset C_5(X,p)$, it follows that $C_4(X,p)$ is a $d$-dimensional linear subspace. Therefore, by Theorem~\ref{thm:main_thm}, the result follows.
\end{proof}

\begin{corollary}\label{cor:c5_linear_var_top}
Let $X\subset \R^n$ be a $d$-dimensional definable set in an o-minimal structure. Then, $X$ is $C^1$ smooth if and only if $X$ is a topological manifold and $C_5(X,p)$ is a $d$-dimensional linear subspace for any $p\in X$.
\end{corollary}
\begin{proof}
We already know from Proposition~\ref{prop:dim_c5} that $\pi|_{X \cap U}:X \cap U \rightarrow \pi(X\cap U)$ is bi-Lipschitz. Therefore, $\pi:X\cap U \rightarrow C_5(X,p)$ is continuous and injective between two topological manifolds of the same dimension. It follows from the Invariance of dimension Theorem that $\pi|_{X\cap U}:X\cap U \rightarrow C_5(X,p)$ is an open mapping. The result then follows from Theorem~\ref{thm:dim_c5}.
\end{proof}

Now, we are ready to state and prove our main result of this subsection.

\begin{theorem}\label{thm:dim_c5_general}
Let $X\subset \R^n$ be a $d$-dimensional definable set in an o-minimal structure. Then, $X$ is $C^1$ smooth if and only if $X$ is a topological manifold and $C_5(X,p)$ is a $d$-dimensional set for any $p\in X$.
\end{theorem}

\begin{proof}
Given $p\in X$, by hypotheses, $C_5(X,p)$ is a $d$-dimensional set. Then there exist a linear projection $\pi:\mathbb{R}^n \rightarrow  \R^d$ such that $\pi^{-1}(0)\cap C_5(X,p)={0}$.
By the same argument used in Proposition $\ref{prop:dim_c5}$, we have that $\pi|_{X \cap U}:X \cap U \rightarrow \pi(X\cap U)$ is bi-Lipschitz, and since $\pi:X\cap U \rightarrow \R^d$ is continuous and injective between topological spaces of the same dimension, it follows from the Domain Invariance Theorem that $\pi$ is open, hence $X$ is Lipschitz regular at $p$.

Since $X$ is $d-$dimensional, we have that $C_4(X,p)$ is a union of $d-$dimensional linear subspaces, but as $C_4(X,p) \subset C_5(X,p)$ and $C_5(X,p)$ is a $d-$dimensional set, it follows that $C_4(X,p)$ is a finite union of $d-$dimensional linear subspaces.

\begin{claim}\label{claim:c3_linear}
     $C_3(X,p)$ is a $d-$dimensional linear subspace for every $p \in X$.
\end{claim}
\begin{proof}[Proof of Claim \ref{claim:c3_linear}]
Hence, $\dim C_5(X,p)=d$, for all $p\in X$. Since $\dim C_4(X,p)\leq \dim C_5(X,p)=d$, $C_4(X,p)$ is a finite union of $d$-dimensional linear subspaces. Then $C_3(X,p)$ is contained in a finite union of $d$-dimensional linear subspaces.
Assume that $C_3(X,p)$ is not contained in a unique $d$-dimensional linear subspace. Then there are $x_1,x_2\in C_3(X,p)$, $\epsilon>0$ and two distinct $d$-dimensional linear subspaces $L_1$ and $L_2$, both contained in $C_4(X,p)$, such that $x_1\in L_1\setminus L_2$, $x_2\in L_2\setminus L_1$, $C_1\cap L_1\subset C_3(X,p)\setminus L_2$ and $C_2\cap L_2\subset C_3(X,p)\setminus L_1$, where $C_i=\{tx;t>0$ and $x\in B_{\epsilon}(x_i)\}$ for $i=1,2$.
Note that $C=\{v-w; v, w\in (C_1\cap L_1)\cup (C_2\cap L_2)\}$ is a subset of $C_5(X,p)$ with $\dim C>d$, which is a contradiction with the fact $\dim C_5(X,p)=d$.
Therefore,
$C_3(X,p)$ is contained in a unique $d$-dimensional linear subspace, let us say $L$. Since $C_3(X,p)$ is lipeomorphic to $L$, then $C_3(X,p)=L$. 
\end{proof}
Now let us show that $C_4(X,p)$ is a $d-$dimensional linear subspace for every $p \in X$. If $p \in Reg_1(X)$, there is nothing to do.
By Proposition V.11 in \cite{Parusinski:2023}, there exists a definable subset $Y \subset X$, $\dim Y \le d - 2$, such that any $p \in X \setminus (Reg_1(X) \cup Y)$ satisfies the following property: There are $m \in \mathbb{N}$ and a neighbourhood $V$ of $p$ such that
\begin{equation*}
   Reg_1(X) \cap V = X_1 \sqcup \dots \sqcup X_m, 
\end{equation*}
with $X_i$ connected, and such that for every $i = 1, \dots, m$, the closure of each $X_i$ in $V$ is a $C^1$ manifold with boundary $V \cap (X \setminus Reg_1(X))$.

Thus, if $p \in X\setminus(Reg_1(X)\cup Y)$, then, since closure $X_i$ is a $C^1$ manifold with boundary and $p \in \partial X_i$, we have that $C_4(X_i,p)$ is a single $d-$dimensional linear subspace, for $1\le i \le m$. But since $C_3(X_i,p)$ is a half space, $C_3(X_i,p)\subset C_4(X_i,p)$, and $C_3(X,p)$ is a $d-$dimensional linear subspace, it follows that $C_3(X,p)= C_4(X_i,p)$, for $1\le i \le m$. Then $C_4(X,p)=\bigcup\limits_{i=1}^m C_4(X_i,p)=C_3(X,p)$. 

By proceeding as in the proof of Claim \ref{claim:c3c4_linear}, we obtain that $C_4(X,p)$ is a $d-$dimensional linear subspace for every $p \in Y$.

Therefore, we have proved that $X$ is Lipschitz regular at any $p\in X$ and $C_4(X,p)$ is a $d$-dimensional linear subspace for any $p\in X$. By Theorem \ref{thm:main_thm}, $X$ is $C^1$ smooth.
\end{proof}

\begin{corollary}\label{cor:dimc5_proj_open}
Let $X\subset \R^n$ be a locally closed $d$-dimensional definable set in an o-minimal structure. Then, $X$ is $C^1$ smooth if and only if for each $p\in X$, $C_5(X,p)$ is a $d$-dimensional subset, and there are a linear projection $\pi \colon \R^n\to \R^d$ and a neighborhood $U$ of $p$ such that $\pi^{-1}(0)\cap C_5(X,p)={0}$ and the restriction of $\pi$ to $X\cap U$ is an open mapping.
\end{corollary}

As an application, we obtain the main result in \cite{KurdykaGN:2018}.

\begin{corollary}[Theorem 4.7 in \cite{KurdykaGN:2018}] \label{cor:c3_c5}
Let $X\subset \R^n$ be a locally closed $d$-dimensional definable set in an o-minimal structure. Then, $X$ is $C^1$ smooth if and only if $C_5(X,p)=C_3(X,p)$ for any $p\in X$.
\end{corollary}
\begin{proof}
Assume $C_5(X,p)=C_3(X,p)$ for any $p\in X$.

It is easy to prove that $C_5(X,p)=C_3(X,p)$ implies that $C_5(X,p)=C_3(X,p)$ is a $d$-dimensional linear subspace.

Fixed $p\in X$, let $\pi\colon \R^n\to \R^d$ be a linear projection such that $\pi^{-1}(0)\cap C_5(X,p)={0}$.
By Proposition \ref{prop:dim_c5}, there is an open set $U$ of $X$ containing $p$ such that $U$ is, in a neighborhood of $p$, the graph of a Lipschitz function defined on $W=\pi(U)$. We claim that $W$ contains an open neighborhood of $\pi(p)\in \R^d$. Indeed, if this is not true, by Proposition V.11 in \cite{Parusinski:2023}, there exists a point $q \in \partial W$ and a neighbourhood $V$ of $q$ such that $V\cap W$ is a $C^1$ manifold with boundary. In particular, $C_3(W,q)$ is a half-space. Since $\phi:=\pi|_U\colon U\to W$ is a definable bi-Lipschitz homeomorphism, by Theorem \ref{inv_cones}, $C_3(W,q)$ and $C(X,\phi^{-1}(q))$ are bi-Lipschitz homeomorphic, which contradicts the facts that $C_3(W,q)$ is a half-space and $C(X,\phi^{-1}(q))$ is a $d$-dimensional linear subspace.

Therefore, $W$ contains an open neighborhood of $\pi(p)\in \R^d$, and thus $X$ is Lipschitz regular at $p$. By Theorem \ref{thm:dim_c5_general}, $X$ is $C^1$ smooth.

The other implication is obvious.
\end{proof}

\subsubsection{Examples}
Let us present and analyse some examples in order to show the sharpness of Theorem \ref{thm:dim_c5_general}.

We cannot drop the condition that $X$ is a topological manifold.
\begin{example}
Let $X=\{(x,y)\in \R^2; y^2-x^3\geq 0\}$. Note that $X$ is not $C^1$ smooth, and we have that $C_5(X,p)=\R^2$ for all $p\in X$.
\end{example}
The next example shows that requiring $C_5(X,p)$ to be a linear subspace only at $p \in X$ does not imply $C^1$ smoothness at $p$.

\begin{example}\label{example:c5_plane_not_smooth}
 Let $X=\{(x,y,z)\in \R^3;z=-y^2(|x|-1)\}$. We have that for each $p\in X$, $X$ is around $p$ the graph of a Lipschitz function. Moreover, $C_5(X,0)$ is the plane $z=0$, and $X$ is not $C^1$ smooth at $0$.
\end{example}

From the proof of Corollary \ref{cor:c3_c5}, it is natural to ask whether $C_5(X,p)$ is always a linear subspace. However, Example \ref{Graf_mod} below shows that it is not.
Even when adding the condition that $X$ is algebraic (hence analytic), the result still fails:

\begin{example}
    Let $X=\{(x,y,z) \in \mathbb{R}^3;x^2y^2+z^2y^2+x^2z^2=0\}$, we have that $C_5(X,0)=\{(x,y,z) \in \mathbb{R}^3;xyz=0\}$, i.e., it is not a vector space. 
\end{example}

Furthermore, even when adding the hypotheses that $X$ is irreducible and a topological manifold, the statement remains false:

\begin{example}
    Consider the set $X=\{(x,y,z)\in \mathbb{R}^3; x^3+y^3-z^3=0\}$; note that $X$ is the graph of a Lipschitz function over the plane $P=\{(x,y,z)\in \mathbb{R}^3;\}$. We will show that $\eta=(1,1,-1) \notin C_5(X,p)$. First, observe that for $p=(p_1,p_2,p_3) \in P$, and with $\pi_{\eta}:X\cap U \rightarrow P$, it follows that $\pi^{-1}(p)\cap X$. We necessarily have that the point satisfies the equation 
    $$(p_1+t)^3+(p_2+t)^3=(p_3-t)^3,$$
    thus obtaining a polynomial of the form $t^3+at+b=0$, where $a>0$. Since its derivative is positive, we have a unique solution, hence $\pi:X \cap U \rightarrow P$ is injective. Now observe that there exists $\epsilon>0$ such that if we project in the direction of $\nu=(1+r_1,1+r_2,-1+r_3)$, $\pi_{\nu}$, where $\|(r_1,r_2,r_3)\|<\epsilon$, we again obtain a polynomial whose coefficients are sufficiently close to those of $t^3+at+b=0$. By the continuity of roots, $\pi_{\nu}^{-1}(p)\cap X$ also has only one real solution, hence $\pi_{\nu}$ is also locally injective. Therefore, there does not exist a secant line passing through two points of $X\cap U$ that approaches $\eta$, so $\eta \notin C_5(X,0)$.     
\end{example}

\section{Smoothness of real analytic sets}\label{sec:reg_analytic}

\begin{theorem}\label{thm:reg_analytic_sets}
Let $X\subset \R^n$ be a $d$-dimensional real analytic set. Then the following items are equivalent:
\begin{enumerate}
    \item $X$ is LNE at any $p\in X$, and $C_4(X,p)$ is a $d$-dimensional linear subspace for all $p\in X$;
    \item $C_5(X,p)$ is a $d$-dimensional subset for all $p\in X$;
    \item $X$ is $C^1$ smooth.
\end{enumerate}
\end{theorem}

\begin{proof}
It is clear that $(3) \Rightarrow (1)$ and $(3) \Rightarrow (2)$.

Now we will show that $(1) \Rightarrow (3)$. So, assume that $X$ is LNE at any $p\in X$, and $C_4(X,p)$ is a $d$-dimensional linear subspace for all $p\in X$. By Theorem \ref{thm:main_thm_one} and Claim \ref{claim:cont_c4}, it is enough to show that $C_3(X,p) = C_4(X,p)$ for all $p \in X$. We may assume that \( p = 0 \), and choose a system of linear coordinates \( (x,y) \) in \( \mathbb{R}^n \) such that \( C_4(X,0) = \{(x,y) \in \mathbb{R}^d\times \mathbb{R}^{n-d} \mid y = 0\} \). Let \( \pi : \mathbb{R}^n \to C_4(X,0) \cong \mathbb{R}^d \) be the orthogonal projection. Since \( X \) is a real analytic set, the multiplicity modulo 2 is well defined. Let \( U \) be a sufficiently small neighborhood of 0 in \( X \) such that there exists a subanalytic set \( \sigma \subset C_4(X,0) \) satisfying $\dim \sigma <d=\dim P$ and \( \#(\pi^{-1}(v) \cap U) \mod 2 \) is constant for all \( v \in C_4(X,0) \setminus \sigma \) sufficiently close to the origin.

Suppose that $C_3(X,p) \subsetneq C_4(X,p)$. Since $C_3(X,p)$ is a closed set, \( W := C_4(X,p) \setminus C_3(X,p) \) is a non-empty open set. For \( v \in W \) and for a sufficiently small neighborhood \( U \) of 0 in \( X \), we have \( \pi^{-1}(tv) \cap U = \emptyset \) for all sufficiently small \( t > 0 \). This shows that the multiplicity of \( X \) at the origin is $0 \mod 2$.
\begin{claim}\label{claim:open_map_2}
    $\pi(U\setminus Sing(X))$ is an open subset of $C_4(X,0)$. 
\end{claim}
\begin{proof}[Proof of Claim \ref{claim:open_map_2}]
Since \( C_4(X,q) \) varies continuously with \( q \), shrinking \( U \) if necessary, we may assume that \( \delta(C_4(X,0), C_4(X,q)) < 1/3 \) for all \( q \in U \). In particular, for each \( q \in U \), the restriction of \( \pi \) to \( C_4(X,q) \) is a linear isomorphism. Thus, \( \pi(U \setminus \text{Sing}(X)) \) is an open subset of \( C_4(X,0) \).    
\end{proof}

 We may assume that \( U = X \cap B^n(0, r) \) for some \( r > 0 \). Let \( B = \pi(U) \), \( B' = B \setminus \pi(\text{Sing}(X)) \), \( U' = \pi^{-1}(B') \cap U \), \( N = \sup_{x \in B'} \#\pi^{-1}_U(x) \), and \( S = \{x \in B' \mid \#\pi^{-1}_U(x) = N\} \). By shrinking r, if necessary, we assume that $0 \in \overline{S}$ and $\pi^{-1}(0)\cap U=\{0\}$. Then \( S \) is an open set, \( \pi^{-1}_U(S) \) has exactly \( N \) connected components, say \( X_1, \ldots, X_N \), and for each \( i \), \( X_i \) is the graph of a \( C^1 \) function \( f_i: S \to \mathbb{R}^{n-d} \). 

Given $z \in S$, the distance from $z$ to $B' \setminus S$ is attained at some $z' \in B' \setminus S$, hence the ball $B_r(z) \subset S$, where $r = \text{dist}(z, B' \setminus S)$. By the assumptions on the tangent cones, each \( f_i \) also has bounded derivative. Let \( \lambda \geq 1 \) such that \( \|Df_i\| \leq \lambda \) on \( S \), for all \( i \in \{1, \ldots, N\} \). Therefore, the functions $f_i:B_r(z) \rightarrow \mathbb{R}^{n-d}$ are lipschitz, so there exists a Lipschitz extension $\tilde{f}:\overline{B_r(z)} \rightarrow \mathbb{R}^{n-d}$. 

Since the multiplicity (mod 2) at the origin is zero, we must have $N \ge 2$, that is, there exist two functions $f_1$ and $f_2$ such that $f_1(z') = f_2(z')$. Since the graph $G_i$ of each $f_i$ is bi-Lipschitz to the closed ball, it follows from Theorem~\ref{inv_cones} that the cone $C_3(G_i,(z',f_i(z')))$ is bi-Lipschitz homeomorphic to the half-space $H = C_3(B_r(z))$. Moreover, 
\[
C_3(G_i,(z',f_i(z'))) = \{(v,D_{+}(f_i)_xv) \mid v \in H\}.
\]
But since $C_4(X,p)$ is a plane, we have $C_4(X,p) = \{(v, L(v)) \mid v \in \mathbb{R}^d\}$, for some linear map $L$, hence $L(v) = D_{+}(f_i)_x v$. 

Now, let $\alpha$ be the line segment joining $z$ to $z'$. Then both $(\alpha(t),f_1(\alpha(t)))$ and $(\alpha(t),f_2(\alpha(t)))$ are tangent, i.e. $\frac{\|(\alpha(t),f_1(\alpha(t)))-(\alpha(t),f_2(\alpha(t)))\|}{t} \to 0$. On the other hand, we have 
\[
d_X((\alpha(t),f_1(\alpha(t))),(\alpha(t),f_2(\alpha(t)))) \ge 2\|\alpha(t)\|,
\]
and we may reparametrize so that $\|\alpha(t)\| = t$, which contradicts the fact that $X$ is LNE.

Now, we are going to prove $(2) \Rightarrow (3)$. So, given $p\in X$, we assume that $C_5(X,p)$ is a $d$-dimensional set. Then there exists a $d$-dimensional linear subspace $P\subset \R^n$ such that the orthogonal projection $\pi\colon \mathbb{R}^n \rightarrow  P$ satisfies $\pi^{-1}(0)\cap C_5(X,p)=\{0\}$.

By the same argument used in Proposition~\ref{prop:dim_c5}, we have that $\pi|_{X \cap U}:X \cap U \rightarrow \pi(X\cap U)$ is bi-Lipschitz, for some open neighborhood $U$ of $p$. 

If we show that $\pi(X\cap U)$ contains an open set, then $X\cap U$, possibly after shrinking the open set, is Lipschitz regular and, in particular, $X$ is a topological manifold. Thus, the result follows from Theorem \ref{thm:dim_c5_general}.

We may assume that \( U \) is taken sufficiently small such that there exists a subanalytic set \( \sigma \subset P \) satisfying $\dim \sigma <d=\dim P$ and \( \#(\pi^{-1}(v) \cap X\cap U) \mod 2 \) is constant for all \( v \in P \setminus \sigma \) sufficiently close to the origin.

Since $\pi:X \cap U \rightarrow \pi(X\cap U)$ is bi-Lipschitz, there exists a Lipschitz extension $\tilde{\pi}:\overline{X\cap U} \rightarrow \overline{\pi(X\cap U)} = \pi(\overline{X\cap U})$. But $X$ is analytic and, in particular, closed in $U$. Hence, $\overline{X\cap U} = X \cap \overline{U}$, and therefore $\pi(X \cap \overline{U}) \subset P$ is closed and contains the origin. 

Thus, $P \setminus \pi(X \cap \overline{U})$ is an open subset of $P$, and so $\dim(P \setminus \pi(X \cap \overline{U})) = \dim(P) > \dim(\sigma)$. 

Now observe that $0 \notin \overline{P \setminus \pi(X \cap \overline{U})}$. Indeed, take a sequence $(x_j)_j \subset (P \setminus \pi(X \cap \overline{U})) \setminus \sigma$ with $x_j \to 0$. Then $\pi^{-1}(x_j) \cap X = \emptyset$, so $m(X,0) = 0 \pmod{2}$, which contradicts the fact that $\pi:X \cap U \rightarrow \pi(X\cap U)$ is injective. 

Therefore, $0 \notin \overline{P \setminus \pi(X \cap \overline{U})}$, so $\pi(X \cap \overline{U})$ contains a neighborhood of the origin.

\end{proof}

\subsection{Examples}

We cannot drop the condition that $X$ is an analytic set.
\begin{example}
Let $X=\{(x,y)\in \R^2; y^2-x^3\leq 0\}$. Note that $X$ is not $C^1$ smooth, and we have that $C_4(X,p)=C_5(X,p)=\R^2$ for all $p\in X$. Moreover, $X$ is LNE at any $p\in X$. 
\end{example}

Note also that there is an analytic set $X$ that is not $C^1$ smooth and is LNE at any $p\in X$, and $C_3(X,p)=C_4(X,p)$ for all $p\in X$.

\begin{example}
For $X=\{(x,y)\in \R^2;xy=0\}$, we have that $C_3(X,p)=C_4(X,p)$ for all $p\in X$. But neither $C_3(X,0)$ nor $C_4(X,0)$  is a linear subspace.
\end{example}

\subsection{Parametrised analytic sets and the answer for Question \ref{question:parametrized}}
We finish this section by proving that Question \ref{question:parametrized} has a negative answer.
\begin{proposition}\label{prop:parametrized_fails}
Let $F\colon \R^2\to \R^3$ be the mapping given by $F(x,y)=(x,y^3,(x+y)^3)$. Then $X=Im(F)$ is an algebraic set, Lipschitz regular at any $p\in X$, $C(X,p)$ is a plane for any $p\in X$, but $X$ is not smooth at $0$.
\end{proposition}
\begin{proof}
Note that $Im(F)=V(f)$, where $f\colon \R^3\to \R$ is given by
$$ f(x,y,z)=(y-z)^3+x^9+3x^6y+3x^3y^2-3x^6z+21x^3yz+3x^3z^2.$$
Thus $X$ is algebraic.
 Now note that $X\setminus\{0\}$ is smooth, so it is enough to analyze only at the origin. Consider the projection $\pi:\mathbb{R}^3 \rightarrow P$ in the direction of the vector $\eta=(1,1,0)$, where $P=      \{(x,y,z) \in \mathbb{R}^3;x+y=0\}$. Notice that this is a bijection when restricted to $X$. Moreover, we can see that by projecting in a neighborhood of the direction $\eta=(1,1,0)$, we also obtain a bijective projection, and therefore the projection is a bi-Lipschitz homeomorphism.

To see that the cone is a linear subspace at every point, simply observe $C_3(X,0)=\{(x,y,z) \in \mathbb{R}^3;y=z\}$. Finally, let us show that $X$ is not $C^1$. For this, note that the normal vectors to the tangent plane at points of the form $(\frac{1}{k},0)$ are vectors of the form $(0,1,0)$, but at $(0,0)$, the normal vector is $(0,1,-1)$. Therefore, the tangent cones, which are planes, do not vary continuously, and therefore $X$ is not $C^1$ smooth.
\end{proof}

\subsection{Zeros of a harmonic function  and the answer for Question \ref{question:harmonic_sets}}
In this Subsection, we show that there are zeros of a harmonic function that are Lipschitz regular but may be non-smooth.

\begin{example}\label{exam:hamornic_set_nonsmooth}
    Let $f:\mathbb{R}^3 \rightarrow \mathbb{R}$ be the homogeneous and harmonic polynomial given by 
    $$f(x,y,z)=x^3+\frac{1}{32}z^3-3xy^2-\frac{3}{64}zx^2-\frac{3}{64}zy^2$$
    Let $X$ be the zero set of $f$. We have that $X$ is a Lipschitz manifold, but it is neither a linear subspace nor is $C^1$ smooth (at 0). If $X$ is $C^1$ smooth (at 0), then $C_3(X,0)$ would be a plane, but $C_3(X,0)=X$, since $f$ is homogeneous. 
    Now let us show that it is a Lipschitz manifold. By the Curve Classification Theorem, there exists a $C^1$ diffeomorphism $h:X\cap \mathbb{S}^2 \rightarrow \mathbb{S}^1$. Since it is of class $C^1$, it is locally Lipschitz, and being defined on a compact set, it is Lipschitz. Therefore, define $H:X\rightarrow \mathbb{R}^2$ by 
    $$
    H(x) = \begin{cases}
        \|x\|h\left(\frac{x}{\|x\|}\right), & \text{if }  x\not=0, \\
        0, & \text{if } x=0.
    \end{cases}
    $$
    Since $H$ is bi-Lipschitz, we obtain the result.
\end{example}

\end{document}